\documentclass[a4paper,usenames,10pt]{amsart}
\usepackage[utf8]{inputenc}
\usepackage[english]{babel}
\usepackage{mathrsfs}
\usepackage{amsthm}
\usepackage{amsfonts}
\usepackage{amsmath}
\usepackage{amssymb}
\usepackage{faktor}
\usepackage{mathtools}
\usepackage{tikz}
\usepackage{esint}
\usepackage{centernot}
\usepackage{verbatim}
\usepackage[a4paper,top=3.5cm,bottom=3.5cm,left=2.5cm,right=2.5cm]{geometry}
\usepackage{setspace}
\usepackage{appendix}
\usepackage{chngcntr}
\usepackage{etoolbox}
\patchcmd{\subsection}{\bfseries}{\itshape}{}{}

\usepackage{cite}
\usepackage{color}
\definecolor{citegreen}{rgb}{0,0.8,0}
\definecolor{refred}{rgb}{0.8,0,0}
\usepackage[colorlinks, citecolor=citegreen, linkcolor=refred]{hyperref}

\newtheorem{theorem}{Theorem}[section]
\newtheorem{lemma}[theorem]{Lemma}
\newtheorem{corollary}[theorem]{Corollary}
\newtheorem{proposition}[theorem]{Proposition}
\theoremstyle{definition}
\newtheorem{remark}[theorem]{Remark}

\newtheorem{definition}[theorem]{Definition}

\newcommand{\R}{\mathbb R}
\numberwithin{equation}{section}
\allowdisplaybreaks
\newcounter{stepctr}
{\end{list}}

\def\XXint#1#2#3{{\setbox0=\hbox{$#1{#2#3}{\int}$}
     \vcenter{\hbox{$#2#3$}}\kern-.5\wd0}}

\DeclareMathOperator{\Vol}{Vol}

\DeclareMathOperator{\Sc}{R}
\DeclareMathOperator{\Ric}{Ric}
\DeclareMathOperator{\Rm}{Rm}
\DeclareMathOperator{\W}{W}
\DeclareMathOperator{\inj}{inj}

\newcommand{\e}{\varepsilon}

\newcommand{\norm}[1]{\| {#1} \|}
\newcommand{\scal}[2]{\langle {#1} , {#2} \rangle}
\newcommand{\mix}[3]{\big\| \| {#1} \|_{{#2}} \big\|_{{#3}}}
\newcommand{\KN}{\mathbin{\bigcirc\mspace{-15mu}\wedge\mspace{3mu}}}

\DeclareUnicodeCharacter{2212}{-}
\title{Mixed Integral Norms for Ricci Flow}
\author{Gianmichele Di Matteo}

\begin{document}
\maketitle
\vspace{-1cm}
\begin{abstract}
We prove that a Ricci flow cannot develop a finite time singularity assuming the boundedness of a suitable space-time integral norm of the curvature tensor. Moreover, the extensibility of the flow is proved under a Ricci lower bound and the boundedness of a space-time integral norm of the scalar curvature.
\end{abstract}
\section{Introduction}
In this paper, we will prove new extension theorems for the Ricci flow. Given a manifold $M$, a family of smooth Riemannian metrics $g(t)$ on $M$ is called a Ricci Flow on the time interval $[0,T) \subset \R$ if it satisfies
\begin{equation}
\frac{\partial g(t)}{\partial t}=-2 \Ric_{g(t)}, \ g(0)=g_0,
\end{equation}
where $\Ric$ denotes the Ricci tensor. By a well-known result of Hamilton, a Ricci flow on a closed manifold $M$ develops a singularity at a finite maximal time $T$ (i.e. the flow cannot be extended past the time $T$) if and only if the maximum of the norm of the Riemannian tensor $\Rm$ blows up at $T$, see \cite{ham1}. This result was extended by Shi \cite{shi} to complete Ricci flows with bounded geometry. Sesum (for closed manifolds) and later Ma and Cheng (for complete manifolds with bounded curvature) showed that indeed a bound on the Ricci curvature rather than the full Riemannian curvature tensor suffices to extend the flow, see \cite{ses,chm}. See also \cite{kot} for a local version of the result. In their paper, Ma and Cheng also proved the extensibility of a complete Ricci flow under the assumption of bounded scalar curvature and Weyl tensor.

A different approach was adopted by Wang in \cite{wan} and consists of considering integral bounds rather than point-wise ones.\emph{Wang's first extension theorem} (Theorem 1.1 in \cite{wan}) states that if  $(M,g(t))$ is a Ricci flow on a closed manifold $M$ satisfying the bound 
\begin{equation}
\norm{\Rm}_{\alpha,M\times[0,T)}<\infty \text{ for some } \alpha \ge \frac{n+2}{2},
\end{equation}
then the flow can be extended past time $T$. We will generalise this result in Theorem \ref{extensionth1} below.

Wang's theorem is proved via a blow-up argument exploiting the scaling invariance of the integral norm above for $\alpha=\frac{n+2}{2}$. His result was extended by Ma and Cheng to complete manifolds in \cite{chm}; similar results were then obtained, for Ricci and Mean Curvature flow, see \cite{hef,le,les,les1,les2,ye,ye1,xuy}. In particular, we remark that Theorem $1.6$ in \cite{les} considers mixed integral norms for the Mean Curvature flow case.

In the same paper, Wang was able to pass from a Riemann curvature integral bound to a scalar curvature one assuming a Ricci lower bound. More precisely, \emph{Wang's second extension theorem} (Theorem 1.2 in \cite{wan}) states that if $(M,g(t))$ is a Ricci flow on a closed manifold $M$ satisfying 
\begin{equation}
\norm{\Sc}_{\alpha,M\times[0,T)}<\infty \text{ for some } \alpha \ge \frac{n+2}{2},
\end{equation}
and $\Ric$ is uniformly bounded from below along the flow up to the time $T$, then the flow can be extended past time $T$. Our Theorem \ref{extensiontheorem2} below is an extension of this result.

Similar results for (compact) Mean Curvature Flow have then been independently obtained by Le and Sesum in \cite{les2} and Xu, Ye and Zhao in \cite{xuy}.

Here we generalise these results using mixed integral norms. For a measurable function $u$, we set
\begin{equation}
\mix{u}{\alpha,\Omega}{\beta,I} \coloneqq \bigg( \int_I{ \bigg( \int_\Omega{|u|^\alpha d\mu_{g(t)}}\bigg)^{\beta/\alpha} dt} \bigg)^{1/\beta}, \label{definitionmixednorm}
\end{equation}
where $\Omega \subseteq M$ and $I \subseteq [0,T)$.
\begin{definition}
Given a couple $(\alpha,\beta)$ of integrability exponents in $(1,\infty)$ and a dimension $n\in \mathbb{N}$, we say that the couple is \textit{optimal} (respectively \textit{super-optimal}, \textit{sub-optimal}) if
\begin{equation}
\alpha=\frac{n}{2}\frac{\beta}{\beta-1} \ \ (\text{resp.} \ \ge, \ \le ).
\end{equation}
\end{definition}
The main reason for introducing the concept of optimal couple is that the mixed integral norm of the Riemann tensor with respect to it is invariant under the parabolic scaling of the Ricci flow.
Our first theorem is a generalization of Wang's first extension theorem to mixed integral norms and complete non compact flows.
\begin{theorem}\label{extensionth1}
Let $(M,g(t))$ be a Ricci flow on a manifold $M$ of dimension $n$, defined on $[0,T)$, $T<+\infty$, and such that $(M,g(t))$ is complete and has bounded curvature for every $t$ in $[0,T)$. Suppose that the initial slice $(M,g(0))$ satisfies $\inj(M,g(0))>0$. Assume the integral bound $\mix{\Rm}{\alpha,M}{\beta,[0,T)}<+\infty$ for some super-optimal couple $(\alpha,\beta)$.
Then the flow can be extended past the time $T$.
\end{theorem}
In this theorem we assumed a control on the geometry of the Ricci flow in order to set up a blow-up procedure near the singular time $T$. In the case the underlying manifold $M$ is closed this control comes for free, so we obtain an extension result under the sole integral bound.
\begin{figure}[h]\label{figure1}
\centering
\includegraphics[scale=1.7]{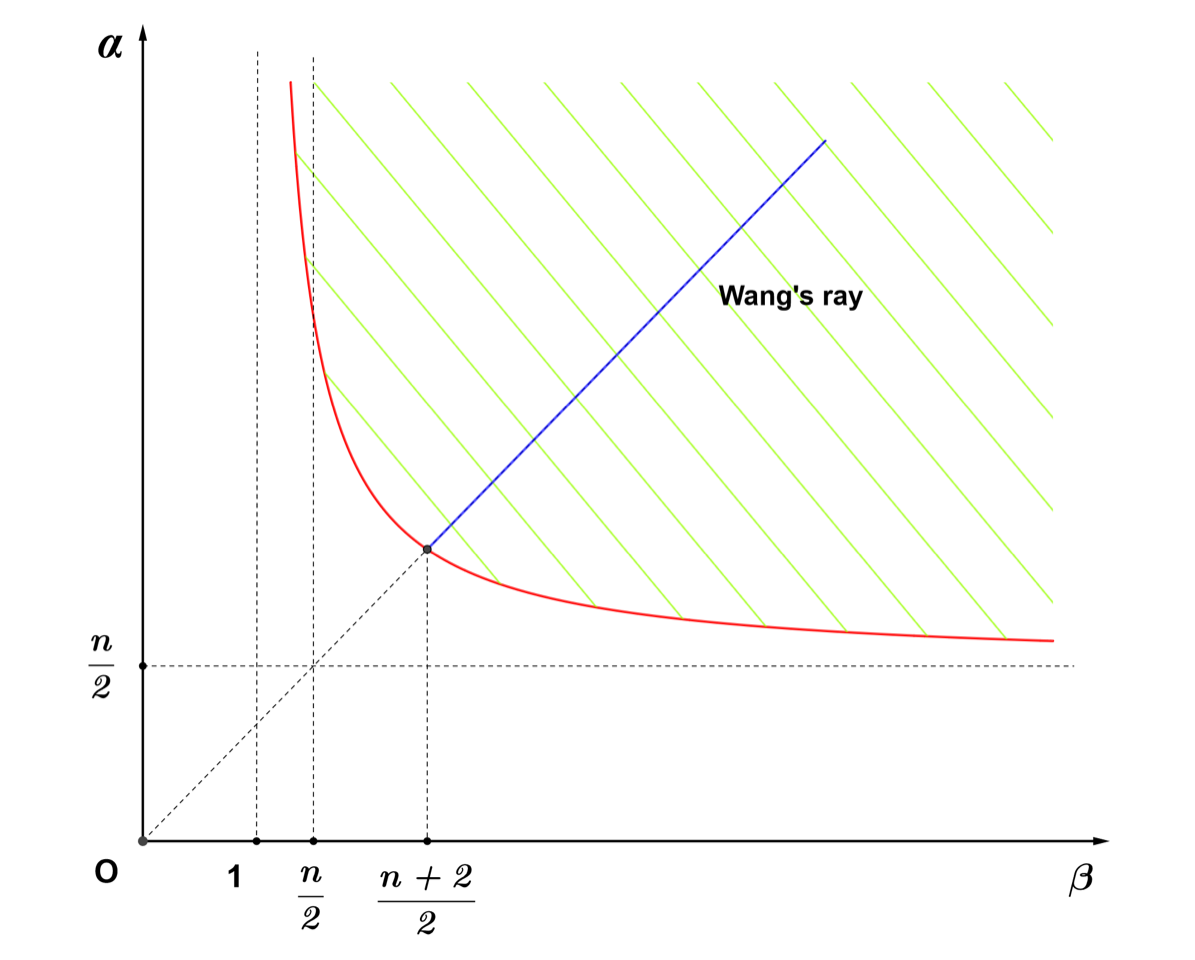}
\caption{Integrability Exponents Graphic}
\end{figure}

It is worth noticing that we can include the ``endpoint" $(\infty,1)$ but not the one $(n/2,\infty)$, even on closed manifolds, see Remark \ref{endpoints}. Moreover, this case is of particular importance after the integrability results obtained in \cite{bam,bam1,sim} for flows with bounded scalar curvature. We generalize Wang's second extension theorem as follows.
\begin{theorem}\label{extensiontheorem2}
Let $(M,g(t))$ be a Ricci flow on a manifold $M$ of dimension $n$, defined on $[0,T)$, $T<+\infty$, and such that $(M,g(t))$ is complete and has bounded curvature for every $t$ in $[0,T)$. Suppose that the initial slice $(M,g(0))$ satisfies $\inj(M,g(0))>0$. Assume the following conditions are satisfied
\begin{itemize}
\item there exists a positive constant $B$ such that $\Ric(x,t) \ge -Bg(t)$ on $M \times [0,T)$;
\item $\mix{\Sc}{\alpha,M}{\beta,[0,T)}<+\infty$ for some super-optimal couple $(\alpha,\beta)$.
\end{itemize}
Then the flow can be extended past the time $T$.
\end{theorem}
As for Wang's result, the method of the proof relies on developing a Moser iteration along the Ricci flow, in order to get a Moser-Harnack inequality for the scalar curvature. Then one can apply this inequality to a rescaled sequence of flows and deduce the result by a contradiction argument.
The main additional difficulty that we have to overcome compared to Wang's case, is that we have to deal with a temporal integrability exponent lower than $n/2$, compare with Figure \ref{figure1}. In order to deal with it, we need to develop a further iteration procedure for showing the better integrability result in Proposition \ref{betterintegrability}.

The paper is organized as follows. In Section \ref{section1}, we recall some standard results in the theory of Ricci flow and we give a proof of Theorem \ref{extensionth1}. In Section \ref{section 2} we first set up the necessary Moser iteration argument and then show Theorem \ref{extensiontheorem2}.
\section*{Acknowledgements}
I would like to express my gratitude to my supervisor Reto Buzano for his patient guidance, constant support and helpful suggestions.
\section{Preliminary Material and First Results}\label{section1}
In this section we recollect some useful results about Ricci flow and we prove some basic results on mixed integral norms in Ricci flow. These results are then combined to prove Theorem \ref{extensionth1}.
\subsection{Standard Results for the Singularity Analysis}
Firstly, we recall Perelman's definition of non local-collapsing and his non local-collapsing theorem, as presented in \cite{kle}. This theorem is particularly important in blow-up arguments, since it yields uniform injectivity radii lower bounds along a sequence of rescaling under mild geometric assumptions.
\begin{definition}[Definition $26.1$ in \cite{kle}]\label{definitioncollapsing}
We say that a Ricci flow solution $(M,g(t))$ defined on $[0,T)$ is \textit{$\kappa-$noncollapsed on the scale $\rho$}, if for every $r<\rho$ and $(x_0,t_0)\in M\times [0,T)$ with $t_0>r^2$, the condition $|\Rm|\le r^{-2}$ on the domain $B_{g(t_0)}(x_0,r)\times[t_0-r^2,t_0]$ implies $\Vol_{g(t)}(B_{g(t_0)}(x_0,r))\ge \kappa r^n$ for every $t\in[t_0-r^2,t_0]$.
\end{definition}
\begin{theorem}[Theorem $26.2$ in \cite{kle}]\label{nonlocalcollapsingtheorem}
For any given $n \in \mathbb{N}$, $T,K<+\infty$, and $\rho,c>0$, there exists $\kappa=\kappa(n,T,\rho,K,c)>0$ such that we have the following. Let $(M,g(t))$ be a Ricci flow on a manifold $M$ of dimension $n$, defined on $[0,T)$, $T<+\infty$, and such that $(M,g(t))$ is complete and has bounded curvature for every $t$ in $[0,T)$. Assume that $(M,g(0))$ is complete, with $|\Rm|_{g(0)}\le K$ and $\inj(M,g(0))\ge c$. Then the flow $g(t)$ is $\kappa-$noncollapsed on the scale $\rho$.
\end{theorem}
Under the same conditions of the theorem, it is sufficient to have the scalar curvature bound $\Sc \le r^{-2}$ to get the non-collapsing of Definition \ref{definitioncollapsing}.
The following compactness theorem, slightly generalising Hamilton's classical compactness result from \cite{ham2}, is taken from Topping \cite{top1}. See also \cite{top2} for an expository review and first applications.
\begin{theorem}[Theorem $1.6$ in \cite{top1}]\label{compactness}
Let $(M_i,g_i(t),p_i)$ be a sequence of pointed and complete Ricci flows, defined on a common time interval $(a,b)$, with $-\infty\le a<0<b\le +\infty$. Suppose that
\begin{itemize}
\item $\inf \inj(M_i,g_i(0),p_i)>0$,
\item there exists a constant $M$ such that for every $r>0$ there exists $i_r \in \mathbb{N}$, such that for every $i\ge i_r$ and $t\in(a,b)$ we have
\begin{equation}
\sup_{B_{g_i(0)}(p_i,r)} |\Rm|_{g_i(t)}\le M.
\end{equation}
\end{itemize}
Then there exists a complete pointed Ricci flow $(M_\infty,g_\infty(t),p_\infty)$ defined on $(a,b)$, which is a pointed smooth Cheeger-Gromov limit of some subsequence of $(M_i,g_i(t),p_i)$.
\end{theorem}
We conclude recalling the orthogonal decomposition for the Riemann tensor
\begin{equation}\label{rmdecomposition}
\Rm=-\frac{\Sc}{2(n-1)(n-2)}g \KN g+\frac{1}{n-2} \Ric \KN g+\W,
\end{equation}
where we denoted by $\W$ the Weyl tensor and by $\KN$ Kulkarni-Nomizu's product of $(0,2)-$tensors.
\subsection{Mixed Integral Norms}
Let us now consider mixed integral norms as defined in (\ref{definitionmixednorm}) along the Ricci flow. Our first result is a lemma that will be used in the Moser's iteration argument.
\begin{lemma}\label{parameterstoinfty}
Let $(M,g(t))$ be a Ricci flow defined in $[0,T]$, and fix a subset $\Omega' \subset M$ such that $0<c\le \Vol_{g(t)}(\Omega')\le C<+\infty$ for every $t$.
Then we have for any measurable $u$
\begin{equation}
\lim_{(a,b) \rightarrow (+\infty,+\infty)}{\mix{u}{a,\Omega'}{b,[0,T]}}=\sup_{\Omega' \times [0,T]}{|u|(x,t)}.
\end{equation}
\end{lemma}
The proof of this lemma directly follows from a standard calculation. For the reader's convenience, we carry this out in the Appendix A.
Before applying Moser's iteration, we need some control on the volume of the domain in consideration. Below we prove a generalization of Property $2.3$ in \cite{wan}.
\begin{lemma}\label{volumebound}
Let $(M,g(t))$ be a Ricci flow defined on $[0,T]$. For a fixed point $p\in M$ and radius  $r$, we set $\Omega \coloneqq B_{g(T)}(p,r)$. Suppose there exists a constant $B>0$ such that $\Ric(x,t)\ge-Bg(t)$ on $\big( \Omega\times[0,T] \big) \cup \big(M \times \{ T\} \big)$. Then there exists a constant $\tilde{V}=\tilde{V}(n,r,T,B)\ge 1$ such that
\begin{equation}
\mix{1}{\alpha,\Omega}{\beta,[0,T]} \le \tilde{V}
\end{equation}
for every $\alpha,\beta \ge 1$. Moreover, $\tilde{V}$ is bounded as long as $n,r,T$ and $B$ remain bounded as well.
\begin{proof}
From the lower Ricci bound on the region $\Omega\times[0,T]$, and the evolution equation for the volume element, we get
\begin{equation}
 \Vol_{g(t)}(\Omega)\le e^{(n-1)BT} \Vol_{g(T)}(\Omega)
\end{equation}
for every $t \in [0,T]$. The lower bound on the whole final time slice allows us to use Bishop-Gromov's inequality, which gives the existence of a constant $V=V(n,r,B)$ such that
\begin{equation}
 \Vol_{g(T)}(\Omega)\le V.
\end{equation}
Therefore, we easily obtain the following chain of inequalities
\begin{equation}
\mix{1}{\alpha,\Omega}{\beta,[0,T]} \le \norm{(e^{(n-1)BT} V)^{\frac{1}{\alpha}}}_{\beta,[0,T]} \le \tilde{V},
\end{equation}
where we have set $\tilde{V}=\max \{e^{(n-1)BT} T^{\frac{1}{\beta}} V^{\frac{1}{\alpha}},1 \}$.
\end{proof}
\end{lemma}
\begin{remark}
It is worth noticing that under rather general assumptions we can take $\tilde{V}=C(n,B)r^{\frac{n}{\alpha}}$ by the results of Zhang \cite{zhaq}, and Chen and Wang \cite{che}.
\end{remark}
An easy application of H\"older's inequality both in space and time yields the following inequality.
\begin{equation}
\int_0^T{\int_M{fg d\mu} dt}\le \mix{f}{\alpha,M}{\beta,[0,T)} \mix{g}{\alpha',M}{\beta',[0,T)},
\end{equation}
where $\alpha'=\alpha/(\alpha-1)$ and $\beta'=\beta/(\beta-1)$ are the H\"older conjugate exponents of $\alpha$ and $\beta$ respectively.
One can use this inequality to produce an interpolation inequality, which we recall.
\begin{proposition}\label{interpolationtheorem}
Let $1\le p\le q \le r<+\infty$, $1\le P\le Q\le R<+\infty$. Suppose that, once written $q=ap+(1-a)r$ and $Q=bP+(1-b)R$ for some $a,b \in[0,1]$, we have the following equations
\begin{equation}
\frac{ap}{q}=\frac{bP}{Q} \eqqcolon s_1 \ \ \text{and} \ \ \frac{(1-a)r}{q}=\frac{(1-b)R}{Q} \eqqcolon s_2.
\end{equation}
Then for every measurable function $v$ we have $\mix{v}{q,M}{Q,[0,T)}\le \mix{v}{p,M}{P,[0,T)}^{s_1} \mix{v}{r,M}{R,[0,T)}^{s_2}$.
\end{proposition}
In Moser's iteration argument we will be interested in extrapolating an optimal couple of exponents given a super-optimal one and the couple $(1,1)$, in the conjugate H\"older exponents plane. The following lemma characterizes such a couple.
\begin{lemma}\label{extrapolationlemma}
For any strictly super-optimal couple $(a,b)$, i.e. such that $a>\frac{n}{2}\frac{b}{b-1}$, there exists a unique optimal couple $(\alpha_*,\beta_*)$ whose image via the conjugate H\"older mapping is extrapolated from $(1,1)$ and $(a',b')$, that is we have
\begin{equation}
\mix{w}{a',M}{b',[0,T)}\le \mix{w}{\alpha_*',M}{\beta_*',[0,T)}^{s_1} \mix{w}{1,M}{1,[0,T)}^{s_2},
\end{equation}
for every measurable function $w$. Moreover, we have
\begin{equation}
\frac{s_2}{1-s_1}=1 \ \ \text{and} \ \ \frac{s_1}{1-s_1}= \frac{b'n(a'-1)+2a'(b'-1)}{2a'-b'n(a'-1)}.
\end{equation}
\begin{proof}
Using the interpolation result Proposition \ref{interpolationtheorem}, we are reduced to show the existence and uniqueness of solution $(\alpha',\beta',\theta,\zeta)$ to the following equations for the exponents:
\begin{equation}
\begin{aligned}
a'=\theta \alpha'+(1-\theta)\cdot 1, \ b'=\zeta \beta'+(1-\zeta)\cdot 1, \ \alpha'=\frac{n\beta'}{n\beta'-2}\\
\frac{\theta\alpha'}{a'}=\frac{\zeta\beta'}{b'}=s_1, \ \frac{(1-\theta)\cdot 1}{a'}=\frac{(1-\zeta)\cdot 1}{b'}=s_2.
\end{aligned}
\end{equation}
The two equations in the second line are easily shown to be equivalent given the others in the first line. Solving the system by simple substitutions, we get
\begin{equation}
\theta=\frac{-2a'+2b'-b'n+a'b'n}{2b'}.
\end{equation}
This $\theta$ is greater than zero because of $a'>1$, and smaller than $1$ because $(a,b)$ is strictly super-optimal. Moreover, $\theta$ determines uniquely the solution to the system, and we have
\begin{equation}
\begin{aligned}
(\alpha^*)'=\frac{b'n(a'-1)+2a'(b'-1)}{b'n(a'-1)-2a'+2b'}, \ (\beta^*)'=\frac{b'n(a'-1)+2a'(b'-1)}{b'n(a'-1)},
\end{aligned}
\end{equation}
from which we deduce
\begin{equation}
\begin{aligned}
s_1=\frac{b'n(a'-1)+2a'(b'-1)}{2a'b'}, \ s_2=\frac{2a'+b'n-a'b'n}{2a'b'} \in(0,1),
\end{aligned}
\end{equation}
and hence the claim.
\end{proof}
\end{lemma}
We conclude this subsection recalling Theorem $4.1$ in Wang's paper \cite{wan}; it regards the existence of a uniform Sobolev constant in a parabolic region, which will play a key role in the argument used in showing Theorem \ref{extensiontheorem2}.
\begin{definition}
We say that a subset $N \subset M$ admits a \textit{uniform Sobolev constant} $\sigma$ at each time slice if
\begin{equation}
\bigg( \int_N{|v|^\frac{2n}{n-2} d\mu_{g(t)}}\bigg)^{\frac{n-2}{n}}\le \sigma \int_N{|\nabla v|^2_{g(t)} d\mu_{g(t)}},
\end{equation}
for every function $v\in W^{1,2}_0(N)$ and $t \in [0,T]$.
\end{definition}
\begin{theorem}[Theorem $4.1$ in \cite{wan}]\label{uniformsobolevth}
Suppose $(M,g(t))$ is a complete Ricci flow defined on $[0,1]$. Fix a point $p\in M$ and suppose that
\begin{itemize}
\item $\Ric(x,t)\ge-(n-1)g(t)$ for every $(x,t)\in M\times [0,1]$;
\item $\Ric(x,t) \le (n-1)g(t)$ for every $(x,t)\in B_{g(1)}(p,1) \times [0,1]$;
\item there exists a constant $\kappa$ such that $\Vol_{g(1)}(B_{g(1)}(p,1))\ge \kappa$.
\end{itemize}
Then there exist a radius $r=r(n,\kappa)$ and a uniform Sobolev constant $\sigma=\sigma(n,\kappa)$ for $B_{g(1)}(p,r(n,\kappa))$ on the time interval $[0,1]$.
\end{theorem}
A different method to develop Sobolev constant bounds along the Ricci flow was obtained by Zhang in \cite{zhaq1}, which had great impact on the study of bounded scalar curvature Ricci flows, for instance see \cite{bam,bam1,sim,sim1}.
\subsection{Proof of Theorem \ref{extensionth1}}
Our proof of Theorem \ref{extensionth1} follows now directly from a blow-up argument, exploiting the scaling behaviour of the norm considered in the statement.
\begin{proof}
If the couple $(\alpha,\beta)$ is super-optimal but not optimal, a straightforward application of H\"older's inequality in time gives the existence of an optimal couple $(\alpha^*,\beta^*)=(\alpha,\beta^*)$, where $\beta^*<\beta$, for which we have
\begin{equation}
\mix{\Rm}{\alpha,M}{\beta^*,[0,T)} \le T^{\frac{1}{(\beta^*)'}}\mix{\Rm}{\alpha,M}{\beta,[0,T)}<\infty,
\end{equation}
so it is sufficient to prove the claim in the optimal case.
Arguing by contradiction, if the flow is not extensible, Shi's Theorem implies that $|\Rm|$ is unbounded on $M \times [0,T)$. From the boundedness assumption for times smaller than $T$, we can pick a sequence of space-time points $(x_i,t_i)$ such that $t_i \nearrow T$, and for some constant $C$ greater than $1$ we obtain
\begin{equation}
|\Rm|(x_i,t_i)\ge C^{-1} \sup_{M\times [0,t_i]} |\Rm|(x,t).
\end{equation}
Set $Q_i \coloneqq |\Rm|(x_i,t_i) \rightarrow +\infty$ and $P_i \coloneqq B_{g(t_i)}(x_i,Q_i^{-\frac{1}{2}}) \times [t_i-Q_i^{-1},t_i]$. Clearly, $|\Rm| \le C Q_i$ on the region $M\times [t_i-Q_i^{-1},t_i]$. Consider a sequence of Ricci flows on $M \times [-Q_i t_i,0]$ defined as $g_i(t)\coloneqq Q_ig(Q_i^{-1}t+t_i)$.
We are in the hypothesis to apply Perelman's $\kappa-$noncollapsing theorem Theorem \ref{nonlocalcollapsingtheorem} for the parabolic region $P_i$, with any scale $\rho$ for $i$ large enough, which guarantees that the injectivity radii of the rescaled metrics $g_i$ at $(x_i,t_i)$ are uniformly bounded away from zero. Then by the compactness result in Theorem \ref{compactness} we can extract a subsequence converging in the pointed smooth Cheeger-Gromov sense to a complete Ricci flow $(M_\infty,g_{\infty}(t),x_{\infty})$ defined on $(-\infty,0]$, whose curvature is uniformly bounded by $C$ and such that $|\Rm_{g_{\infty}}|(x_{\infty},0)=1$.
On the other hand, if the couple $(\alpha,\beta)$ is optimal, we compute
\begin{equation}
\begin{aligned}
&\int_{-1}^0{ \bigg( \int_{B_{g_{\infty}(0)}(x_\infty,1)}{|\Rm_{g_{\infty}(t)}|^{\alpha} d\mu_{g_{\infty}(t)}}\bigg)^{\frac{\beta}{\alpha}} dt} \\
= \lim_{i \rightarrow \infty} &\int_{-1}^0{ \bigg( \int_{B_{g_i(0)}(\bar{x}_i,1)}{|\Rm_{g_i(t)}|^{\frac{n}{2} \frac{\beta}{\beta-1}} d\mu_{g_i(t)}}\bigg)^{\frac{2}{n}(\beta-1)} dt}\\
=\lim_{i \rightarrow \infty} &\int_{t_i-Q_i^{-1}}^{t_i}{ \bigg( \int_{B_{g(t_i)}(x_i,Q_i^{-\frac{1}{2}})}{|\Rm_{g(t)}|^{\frac{n}{2} \frac{\beta}{\beta-1}} Q_i^{\frac{n}{2}-\frac{n}{2} \frac{\beta}{\beta-1}} d\mu_{g(t)}}\bigg)^{\frac{2}{n}(\beta-1)} Q_idt}\\
\le \lim_{i \rightarrow \infty} &\int_{t_i-Q_i^{-1}}^{t_i}{ \bigg( \int_{M}{|\Rm_{g(t)}|^{\frac{n}{2} \frac{\beta}{\beta-1}} d\mu_{g(t)}}\bigg)^{\frac{2}{n}(\beta-1)} dt}=0,
\end{aligned}
\end{equation}
where the last step is justified by Lebesgue's dominated convergence theorem and the assumption $\mix{\Rm}{\alpha,M}{\beta,[0,T)}<\infty$. Since the limit flow $g_{\infty}(t)$ is smooth, this chain of inequalities implies that $\Rm_{g_{\infty}(t)}\equiv 0$ on the parabolic region $B_{g_{\infty}(0)}(x_{\infty},1)\times [-1,0]$, in particular, $\Rm_{\infty}(x_{\infty},0)=0$, a contradiction.
\end{proof}
\begin{remark}\label{endpoints}
It is interesting to analyse the "endpoints" case.

Firstly, consider $\alpha=\infty$ and $\beta=1$. Even in the closed case, Hamilton's theorem in \cite{ham1} guarantees that the sectional curvature blows up at the finite time singularity $T$, and a maximum principle argument yields $|\Rm| \ge \frac{1}{8(T-t)} \notin L^1$. Moreover, the boundedness of the $L^1$-norm of the maximum of the Ricci curvature is sufficient to extend the flow, as shown in \cite{wan} and subsequently in \cite{hef}.

In the case $\alpha=\frac{n}{2}$ and $\beta=\infty$, the Ricci flow of the standard sphere shows that one cannot expect to extend the flow even if $\mix{\Rm}{\frac{n}{2},S^n}{\infty,[0,T)}<+\infty$. Interestingly, an extension theorem is proven in \cite{ye} under a \textit{smallness assumption} on the $(n/2,\infty)$-mixed norm.
\end{remark}
\begin{corollary}\label{riemannclosedcase}
Let $(M,g(t))$ be a Ricci flow on a closed manifold $M$ of dimension $n$, defined on $[0,T)$, with $T<+\infty$. Assume the integral bound $\mix{\Rm}{\alpha,M}{\beta,[0,T)}<+\infty$ for some super-optimal couple $(\alpha,\beta)$.
Then the flow can be extended past the time $T$.
\end{corollary}
The following result generalizes Theorem $1.1$ in \cite{chm} to mixed norms along complete (possibly non-compact) Ricci flows. The proof strictly follows the one in \cite{chm}.
\begin{theorem}
Let $(M,g(t))$ be a Ricci flow on a manifold $M$ of dimension $n$, defined on $[0,T)$, $T<+\infty$, and such that $(M,g(t))$ is complete and has bounded curvature for every $t$ in $[0,T)$. Suppose that the initial slice $(M,g(0))$ satisfies $\inj(M,g(0))>0$. Assume the integral bounds $\mix{\Sc}{\alpha,M}{\beta,[0,T)}<+\infty$ and $\mix{\W}{\alpha,M}{\beta,[0,T)}<+\infty$ for some super-optimal couple $(\alpha,\beta)$.
Then the flow can be extended past the time $T$.
\begin{proof}
Arguing by contradiction, if the flow is not extensible, Shi's Theorem implies that $|\Rm|$ is unbounded on $M \times [0,T)$. From the boundedness assumption for times smaller than $T$, we can pick a sequence of space-time points $(x_i,t_i)$ such that $t_i \nearrow T$, and for some constant $C$ greater than $1$ we have
\begin{equation}
|\Rm|(x_i,t_i)\ge C^{-1} \sup_{M\times [0,t_i]} |\Rm|(x,t).
\end{equation}
Set $Q_i \coloneqq |\Rm|(x_i,t_i) \rightarrow +\infty$ and $P_i \coloneqq B_{g(t_i)}(x_i,Q_i^{-\frac{1}{2}}) \times [t_i-Q_i^{-1},t_i]$. Clearly, $|\Rm| \le C Q_i$ on the region $M\times [t_i-Q_i^{-1},t_i]$. Consider a sequence of Ricci flows on $M \times [-Q_i t_i,0]$ defined as $g_i(t)\coloneqq Q_ig(Q_i^{-1}t+t_i)$.
We can argue as in the proof of Theorem \ref{extensionth1} to extract a subsequence converging in the pointed smooth Cheeger-Gromov sense to a complete Ricci flow $(M_\infty,g_{\infty}(t),x_{\infty})$ defined on $(-\infty,0]$, whose curvature is uniformly bounded by $C$ and such that $|\Rm_{g_{\infty}}|(x_{\infty},0)=1$.
Again, if the couple $(\alpha,\beta)$ is optimal, the scaling properties of $\Sc$ and $\W$ and the finiteness of their mixed integral norms give us
\begin{equation}
\int_{-1}^0{ \bigg( \int_{B_{g_{\infty}(0)}(x_\infty,1)}{|\Sc_{g_{\infty}(t)}|^{\alpha} d\mu_{g_{\infty}(t)}}\bigg)^{\frac{\beta}{\alpha}} dt}=0,
\end{equation}
and
\begin{equation}
\int_{-1}^0{ \bigg( \int_{B_{g_{\infty}(0)}(x_\infty,1)}{|\W_{g_{\infty}(t)}|^{\alpha} d\mu_{g_{\infty}(t)}}\bigg)^{\frac{\beta}{\alpha}} dt}=0.
\end{equation}
Once more we deduce from the smoothness of the limit flow $g_{\infty}(t)$, together with these equations, that $\Sc_{g_{\infty}(t)}\equiv 0$, thus also $\Ric_{g_{\infty}(t)}\equiv 0$ from the evolution equation of the scalar curvature, and $\W_{g_{\infty}(t)}\equiv 0$ on the parabolic region $B_{g_{\infty}(0)}(x_{\infty},1)\times [-1,0]$, from which we deduce $\Rm_{g_{\infty}(t)}\equiv 0$ through (\ref{rmdecomposition}); in particular, $\Rm_{g_{\infty}}(x_{\infty},0)=0$, a contradiction. We argue exactly as in the proof of Theorem \ref{extensionth1} in the case $(\alpha,\beta)$ is super-optimal but not optimal. 
\end{proof}
\end{theorem}
\section{Parabolic Moser Iteration and Proof of Theorem \ref{extensiontheorem2}}\label{section 2}
In this section we prove Theorem \ref{extensiontheorem2}. The idea of the proof is similar to the one of Theorem \ref{extensionth1} in the previous section, but this time we will rescale with scalar curvature rather than Riemannian curvature. Consequently, we do not have the full curvature bounds needed to extract a smooth limit flow. Hence we will need to prove the necessary estimates for elements of the sequence of rescaled flows rather than for the limit. For this reason, we develop a Moser iteration along the Ricci flow in order to obtain a Moser-Harnack type inequality. This is the main technical part of this article. The main method of the proof resembles the one in Wang's paper \cite{wan}; however, several modifications are necessary. We first settle the super-optimal case (Theorem \ref{supermoser}) and then we prove the optimal case (Theorem \ref{optimalmoser}) with the help of a better integrability result (Theorem \ref{betterintegrability}). The main difficulty to overcome - which is not present in Wang's case - is given by the possibly low temporal integrability case, when $\beta<\frac{n}{2}$. We will deal with this constructing an iterative scheme of reverse H\"older inequalities.
\subsection{Moser Iteration in the Super-optimal Case}
Throughout this section, we consider a fixed complete Ricci flow $(M,g(t))$ on a $n-$dimensional manifold $M$, with $n\ge 3$, defined on $[0,T]$.
\begin{definition}
For any given point $p \in M$ and radius $r>0$, we define the sets
\begin{equation}
\begin{aligned}
\Omega \coloneqq B_{g(T)}(p,r),\ \ \  \Omega' \coloneqq B_{g(T)}\Big(p,\tfrac{r}{2} \Big),\\
D \coloneqq \Omega \times [0,T],\ \ \ D' \coloneqq \Omega' \times \Big[\tfrac{T}{2},T \Big].
\end{aligned}
\end{equation}
\end{definition}
From now on we suppose to have a uniform Sobolev constant for the domain $\Omega$, and also the bound $0<\Vol_{g(t)}(\Omega')<+\infty$  for $t \in [\frac{T}{2},T]$. We have the following analogue of Property $4.1$ in Wang's paper \cite{wan}.
\begin{lemma}
Under the above assumptions, consider a function $v \in C^1(D)$ with $v(\cdot,t) \in C^1_0(\Omega)$ for every $t \in [0,T]$. Then for any $(\alpha,\beta)$ optimal couple we have
\begin{equation}
\mix{v^2}{\alpha',\Omega}{\beta',[0,1]} \le \sigma^{\frac{1}{\beta'}} \max_{t\in [0,1]}{\norm{v(\cdot,t)}_{2,\Omega}^{s}} \bigg( \int_D{|\nabla v|^2 d\mu}\bigg)^{\frac{1}{\beta'}}=\sigma^{\frac{1}{\beta'}} \mix{v}{2,\Omega}{\infty,[0,1]}^{s} \norm{\nabla v}_{2,D}^{\frac{2}{\beta'}},
\label{property}
\end{equation}
where $s \coloneqq \frac{\alpha'(2-n)+n}{\alpha'} \in (0,2)$.
\begin{proof}
For the convenience of the reader, we remark that
\begin{equation}
\alpha'=\frac{\alpha}{\alpha-1}=\frac{n\beta'}{n\beta'-2},\ \ \ \beta'=\frac{\beta}{\beta-1}=\frac{2\alpha'}{n(\alpha'-1)}.
\end{equation}
Moreover, if we set $a \coloneqq n(\alpha'-1)$, $b\coloneqq \alpha'(2-n)+n$, $p\coloneqq \frac{2}{(n-2)(\alpha'-1)}$ and $q\coloneqq p'= \frac{2}{\alpha'(2-n)+n}$, we have
\begin{equation}
\begin{aligned}
a+b=2\alpha', \ \ ap=\frac{2n}{n-2}=2^*, \ \ bq=2, \ \ \frac{1}{p}+\frac{1}{q}=1, \ \ \frac{a\beta'}{2\alpha'}=1,\ \ \frac{2}{q \alpha'}=s.
\end{aligned}
\end{equation}
Therefore, we compute using subsequently H\"older and Sobolev inequalities
\begin{align}
\mix{v^2}{\alpha',\Omega}{\beta',[0,T]}=\bigg( \int_0^T{ \bigg( \int_\Omega {|v|^{2\alpha'} d\mu}\bigg)^{\frac{\beta'}{\alpha'}} dt} \bigg)^{\frac{1}{\beta'}}=\bigg( \int_0^T{ \bigg( \int_\Omega {|v|^{a+b} d\mu}\bigg)^{\frac{\beta'}{\alpha'}} dt} \bigg)^{\frac{1}{\beta'}}\\
\le \bigg( \int_0^T{ \bigg( \int_\Omega {|v|^{ap} d\mu}\bigg)^{\frac{1}{p} \frac{\beta'}{\alpha'}} \bigg( \int_\Omega {|v|^{bq} d\mu}\bigg)^{\frac{1}{q} \frac{\beta'}{\alpha'}} dt} \bigg)^{\frac{1}{\beta'}}=\bigg( \int_0^T{ \bigg( \int_\Omega {|v|^{2^*} d\mu}\bigg)^{\frac{a\beta'}{2^* \alpha'}} \bigg( \int_\Omega {|v|^2 d\mu}\bigg)^{\frac{\beta'}{q\alpha'}} dt} \bigg)^{\frac{1}{\beta'}} \nonumber \\
\le \bigg( \int_0^T{ \bigg( \sigma \int_\Omega {|\nabla v|^2 d\mu}\bigg)^{\frac{a\beta'}{2\alpha'}} \bigg( \int_\Omega {|v|^2 d\mu}\bigg)^{\frac{\beta'}{q \alpha'}} dt} \bigg)^{\frac{1}{\beta'}}=\bigg( \int_0^T{ \bigg( \sigma \int_\Omega {|\nabla v|^2 d\mu}\bigg) \bigg( \int_\Omega {|v|^2 d\mu}\bigg)^{\frac{\beta'}{q \alpha'}} dt} \bigg)^{\frac{1}{\beta'}} \nonumber \\
\le \sigma^{\frac{1}{\beta'}} \bigg( \int_D{|\nabla v|^2 d\mu dt} \bigg)^{\frac{1}{\beta'}} \cdot \max_{t\in[0,T]} \norm{v(\cdot,t)}_{2,\Omega}^{\frac{\beta'}{q \alpha'} \frac{2}{\beta'}}=\sigma^{\frac{1}{\beta'}} \bigg( \int_D{|\nabla v|^2 d\mu dt} \bigg)^{\frac{1}{\beta'}} \cdot \max_{t\in[0,T]} \norm{v(\cdot,t)}^s_{2,\Omega}.\nonumber
\end{align}
\end{proof}
\end{lemma}
We start showing a Moser-Harnack inequality for super-optimal couples.
\begin{lemma}\label{supermoser}
Given $(M,g(t))$ as above, suppose there exists a constant $B \ge 0$ such that on $D$ we have $\Ric(x,t) \ge -Bg(t)$ and let $(a,b)$ be a strictly super-optimal couple. Assume that for two measurable functions $f$ and $h$ there exists a non-negative function $u \in C^\infty(D)$ satisfying
\begin{equation}
\frac{\partial u}{\partial t} \le \Delta u+fu+h
\label{equation1}
\end{equation}
in the sense of distributions, where $\mix{f}{a,\Omega}{b,[0,T]}+\mix{\Sc_{-}}{a,\Omega}{b,[0,T]}+1 \le C_0$. Then there exists a constant $C=C(n,a,b,\sigma,C_0,r,T,B)$ such that
\begin{equation}\label{equation2}
\norm{u}_{\infty,D'} \le C(\mix{u}{\alpha',\Omega}{\beta',[0,T]}+\mix{h}{a,\Omega}{b,[0,T]}\cdot \mix{1}{\alpha',\Omega}{\beta',[0,T]}),
\end{equation}
where $\alpha=\alpha_*(a,b,n)$ and $\beta=\beta_*(a,b,n)$ are the optimal integrability exponents given by Lemma \ref{extrapolationlemma}.
\begin{proof}
Consider a cut-off function $\eta \in C^{\infty}(D)$ such that $\eta(\cdot,t)\in C^{\infty}_0(\Omega)$ for every $t \in [0,T]$, $\eta(x,0)\equiv 0$ and $\eta(x,\cdot)$ is a non-decreasing function for every $x \in \Omega$. Set $\kappa \coloneqq \mix{h}{a,\Omega}{b,[0,T]}$ and $v\coloneqq u+\kappa$. Rewriting (\ref{equation1}) in terms of $v$ we simply have
\begin{equation}
\frac{\partial v}{\partial t} -\Delta v \le f(v-\kappa)+h.
\end{equation}
For a fixed $\lambda>1$, it makes sense to consider $\eta^2(u+\kappa)^{\lambda-1}$ as a test function, so we get for any $s \in (0,T]$
\begin{equation}
\begin{aligned}
\int_0^s \int_\Omega{(-\Delta v)\eta^2 v^{\lambda-1} d\mu dt}+\int_0^s \int_\Omega{\frac{\partial v}{\partial t}\eta^2 v^{\lambda-1} d\mu dt}&\le \int_0^s \int_\Omega{(fu+h)\eta^2 (u+\kappa)^{\lambda-1} d\mu dt} \\
&\le \int_0^s \int_\Omega{\Big(|f|+\frac{|h|}{\kappa}\Big)\eta^2 v^{\lambda} d\mu dt}.
\end{aligned}
\end{equation}
Using the equation for the volume element under the Ricci flow and integrating by parts, we deduce
\begin{equation}
\begin{aligned}
 \int_0^s& \int_\Omega{(2\eta \scal{\nabla \eta}{\nabla v} v^{\lambda-1} +(\lambda-1)\eta^2 v^{\lambda-2}|\nabla v|^2)d\mu dt}\\
 +&\frac{1}{\lambda} \bigg( \int_\Omega \eta^2 v^\lambda d\mu \bigg\rvert_s-\int_0^s \int_\Omega{2\eta\frac{\partial \eta}{\partial t} v^{\lambda} d\mu dt}+\int_0^s \int_\Omega{\eta^2 v^{\lambda} \Sc d\mu dt}\bigg)\le \int_0^s \int_\Omega{\Big(|f|+\frac{|h|}{\kappa}\Big)\eta^2 v^{\lambda} d\mu dt}.
\end{aligned}
\end{equation}
Schwartz's inequality yields the following estimate
\begin{equation}
\int_0^s \int_\Omega{(2\eta \scal{\nabla \eta}{\nabla v} v^{\lambda-1}d\mu dt}\ge -\e^2 \int_0^s \int_\Omega{\eta^2 v^{\lambda-2}|\nabla v|^2d\mu dt}-\frac{1}{\e^2}\int_0^s \int_\Omega{v^{\lambda}|\nabla \eta|^2d\mu dt}.
\end{equation}
Substituting in the previous one we obtain, after reordering, that
\begin{equation}
\begin{aligned}
(\lambda-1-\e^2)&\int_0^s \int_\Omega{\eta^2 v^{\lambda-2}|\nabla v|^2d\mu dt}+\frac{1}{\lambda} \int_\Omega \eta^2 v^\lambda d\mu \bigg\rvert_s
\le\int_0^s \int_\Omega{\Big(|f|+\frac{|h|}{\kappa}\Big)\eta^2 v^{\lambda} d\mu dt}\\
&+\frac{1}{\e^2}\int_0^s \int_\Omega{v^{\lambda}|\nabla \eta|^2 d\mu dt}+\frac{1}{\lambda} \bigg(\int_0^s \int_\Omega{2\eta\frac{\partial \eta}{\partial t} v^{\lambda} d\mu dt}-\int_0^s \int_\Omega{\eta^2 v^{\lambda} \Sc d\mu dt}\bigg).
\end{aligned}
\end{equation}
Choose $\e^2=\frac{\lambda-1}{2}$. Since $|\nabla v^{\frac{\lambda}{2}}|^2=\frac{\lambda^2}{4}v^{\lambda-2}|\nabla v|^2$, we compute
\begin{equation}
\begin{aligned}
2\Big( 1-\frac{1}{\lambda}\Big)&\int_0^s \int_\Omega{\eta^2 |\nabla v^{\frac{\lambda}{2}}|^2 d\mu dt}+\int_\Omega \eta^2 v^\lambda d\mu \bigg\rvert_s \le \lambda \int_0^s \int_\Omega{\Big(|f|+\frac{|h|}{\kappa}+\Sc_{-} \Big)\eta^2 v^{\lambda} d\mu dt}\\
&+\frac{2\lambda}{\lambda-1}\int_0^s \int_\Omega{v^{\lambda}|\nabla \eta|^2d\mu dt}+\int_0^s \int_\Omega{2\eta\frac{\partial \eta}{\partial t} v^{\lambda} d\mu dt}.
\end{aligned}
\end{equation}
Now we use $|\nabla(\eta v^{\frac{\lambda}{2}})|^2\le 2\eta^2 |\nabla v^{\frac{\lambda}{2}}|^2+2v^\lambda |\nabla \eta|^2$ to infer that
\begin{equation}
\begin{aligned}
\Big( 1-\frac{1}{\lambda}\Big)&\int_0^s \int_\Omega{|\nabla(\eta v^{\frac{\lambda}{2}})|^2 d\mu dt}+\int_\Omega \eta^2 v^\lambda d\mu \bigg\rvert_s \le \lambda \int_0^s \int_\Omega{\Big(|f|+\frac{|h|}{\kappa}+\Sc_{-} \Big)\eta^2 v^{\lambda} d\mu dt}\\
&+2\Big(\frac{\lambda}{\lambda-1}+\frac{\lambda-1}{\lambda}\Big)\int_0^s \int_\Omega{v^{\lambda}|\nabla \eta|^2d\mu dt}+\int_0^s \int_\Omega{2\eta\frac{\partial \eta}{\partial t} v^{\lambda} d\mu dt}.
\end{aligned}
\end{equation}
Therefore we have
\begin{align}
\int_0^s \int_\Omega|\nabla(\eta v^{\frac{\lambda}{2}})|^2 d\mu dt+\int_\Omega \eta^2 v^\lambda d\mu \bigg\rvert_s \nonumber\\
\le \Lambda(\lambda) \bigg( \int_0^s \int_\Omega{\Big(|f|+\frac{|h|}{\kappa}+\Sc_{-} \Big)\eta^2 v^{\lambda} d\mu dt}+\int_0^s \int_\Omega{v^{\lambda}|\nabla \eta|^2d\mu dt}+&\int_0^s \int_\Omega{2\eta\frac{\partial \eta}{\partial t} v^{\lambda} d\mu dt}\bigg) \\
\le \Lambda(\lambda) \bigg( \mix{|f|+\frac{|h|}{\kappa}+\Sc_{-}}{a,\Omega}{b,[0,T]} \mix{\eta^2 v^{\lambda}}{a',\Omega}{b',[0,T]}+\int_0^s \int_\Omega{v^{\lambda}|\nabla \eta|^2d\mu dt}+&\int_0^s \int_\Omega{2\eta\frac{\partial \eta}{\partial t} v^{\lambda} d\mu dt}\bigg) \nonumber\\
\le \Lambda(\lambda) \bigg( C_0 \mix{\eta^2 v^{\lambda}}{a',\Omega}{b',[0,T]}+\int_0^s \int_\Omega{v^{\lambda}|\nabla \eta|^2d\mu dt}+&\int_0^s \int_\Omega{2\eta\frac{\partial \eta}{\partial t} v^{\lambda} d\mu dt}\bigg).\nonumber
\end{align}
The constant $\Lambda(\lambda)$ can be chosen as follows:
\begin{equation}
\Lambda(\lambda)=
\begin{cases}
4 \frac{\lambda^2}{(\lambda-1)^2} \ \ \ \text{if} \ 1<\lambda <2,\\
6\lambda \ \ \ \text{if} \ \lambda \ge 2.
\end{cases}
\end{equation}
In particular, we get
\begin{equation}
\begin{aligned}
\int_0^T \int_\Omega{|\nabla(\eta v^{\frac{\lambda}{2}})|^2 d\mu dt}&\le \Lambda(\lambda) \bigg( C_0 \mix{\eta^2 v^{\lambda}}{a',\Omega}{b',[0,T]}+\left\|\Big(|\nabla \eta|^2+2\eta\frac{\partial \eta}{\partial t} \Big) v^{\lambda} \right\|_{1,D} \bigg),\\
\max_{0\le s \le T}\int_\Omega \eta^2 v^\lambda d\mu \bigg\rvert_s &\le \Lambda(\lambda) \bigg( C_0 \mix{\eta^2 v^{\lambda}}{a',\Omega}{b',[0,T]}+\left\|\Big(|\nabla \eta|^2
+2\eta\frac{\partial \eta}{\partial t} \Big) v^{\lambda} \right\|_{1,D} \bigg).
\end{aligned}
\end{equation}
Applying the inequality (\ref{property}) to the function $w=\eta v^{\frac{\lambda}{2}}$,  we arrive to
\begin{equation}\label{moserpreabsorbed}
 \mix{\eta^2 v^{\lambda}}{\alpha',\Omega}{\beta',[0,T]} \le \sigma^{\frac{1}{\beta'}}\Lambda(\lambda) \bigg( C_0 \mix{\eta^2 v^{\lambda}}{a',\Omega}{b',[0,T]}+\left\|\Big(|\nabla \eta|^2
+2\eta\frac{\partial \eta}{\partial t} \Big) v^{\lambda} \right\|_{1,D} \bigg),
\end{equation}
where we have chosen the optimal integrability couple $(\alpha,\beta)=(\alpha^*,\beta^*)$ given by Lemma \ref{extrapolationlemma}. Coherently to the notation in that Lemma, we have used that
\begin{equation}
s+\frac{2}{\beta'}=\frac{\alpha'(2-n)+n}{\alpha'}+\frac{2}{\beta'}=\frac{\alpha'(2-n)+n+n(\alpha'-1)}{\alpha'}=2.
\end{equation}
Combining Lemma \ref{extrapolationlemma} and Young's inequality, we obtain
\begin{equation}
\begin{aligned}
\mix{\eta^2 v^{\lambda}}{a',\Omega}{b',[0,T]} &\le \mix{\eta^2 v^{\lambda}}{\alpha',\Omega}{\beta',[0,T]}^{s_1} \mix{\eta^2 v^{\lambda}}{1,\Omega}{1,[0,T]}^{s_2}\\
&\le \mix{\eta^2 v^{\lambda}}{\alpha',\Omega}{\beta',[0,T]} \e^{s_1} s_1+ \mix{\eta^2 v^{\lambda}}{1,\Omega}{1,[0,T]}^{\frac{s_2}{1-s_1}} \e^{-\frac{1}{1-s_1}}(1-s_1) \\
&\le  \mix{\eta^2 v^{\lambda}}{\alpha',\Omega}{\beta',[0,T]} \e^{s_1}+ \mix{\eta^2 v^{\lambda}}{1,\Omega}{1,[0,T]}^{\frac{s_2}{1-s_1}} \e^{-\frac{1}{1-s_1}}\\
&=\mix{\eta^2 v^{\lambda}}{\alpha',\Omega}{\beta',[0,T]} \e'+ \mix{\eta^2 v^{\lambda}}{1,\Omega}{1,[0,T]} \e'^{-\frac{s_1}{1-s_1}\eqqcolon -\nu},
\end{aligned}
\end{equation}
where $\e'=\e^{s_1}$. This allows us absorbing the first term in the right hand side of the inequality (\ref{moserpreabsorbed}) to the left hand side:
\begin{equation}
\big(1-\Lambda(\lambda)\sigma^{\frac{1}{\beta'}}C_0\e' \big) \mix{\eta^2 v^{\lambda}}{\alpha',\Omega}{\beta',[0,T]} \le \sigma^{\frac{1}{\beta'}}\Lambda(\lambda) \big( C_0 \e'^{-\nu}\norm{\eta^2 v^{\lambda}}_{1,D}+\big\|\big(|\nabla \eta|^2
+2\eta \partial_t \eta \big) v^{\lambda} \big\|_{1,D} \big).
\end{equation}
Setting $\e'=(2\Lambda(\lambda)\sigma^{\frac{1}{\beta'}}C_0)^{-1}$, we have
\begin{equation}
\mix{\eta^2 v^{\lambda}}{\alpha',\Omega}{\beta',[0,T]} \le 2\sigma^{\frac{1}{\beta'}}\Lambda(\lambda) \bigg( C_0 \big( 2\Lambda(\lambda)\sigma^{\frac{1}{\beta'}}C_0 \big)^{\nu}\norm{\eta^2 v^{\lambda}}_{1,D}+\left\|\Big(|\nabla \eta|^2
+2\eta\frac{\partial \eta}{\partial t} \Big) v^{\lambda} \right\|_{1,D} \bigg).
\end{equation}
Since we can always choose $\Lambda(\lambda)\ge 1$, we get
\begin{equation}
\mix{\eta^2 v^{\lambda}}{\alpha',\Omega}{\beta',[0,T]} \le C_1(n,a,b,\sigma,C_0) \Lambda(\lambda)^{1+\nu} \int_D{\Big(|\nabla \eta|^2
+2\eta\frac{\partial \eta}{\partial t}+\eta^2 \Big) v^{\lambda} d\mu dt}.
\label{moserabsorbed}
\end{equation}
The inequality (\ref{moserabsorbed}) is the basis for the Moser iteration. As in Wang's paper \cite{wan}, we will construct a nested sequence of cylindrical sets and test functions on them, and we will eventually arrive at an $L^\infty-$bound. 
Define for every natural number $k$
\begin{equation}
\begin{aligned}
t_k \coloneqq \frac{T}{2}-\frac{T}{4^{k+1}},& \ \ \ r_k \coloneqq \Big( \frac{1}{2}+\frac{1}{2^{k+1}} \Big)r,\\
\Omega_k\coloneqq B_{g(T)}(p,r_k),& \ \ \ D_k \coloneqq \Omega_k \times [t_k,T].
\end{aligned}
\end{equation}
Notice that $t_k \nearrow \frac{T}{2}$, $r_k \searrow \frac{1}{2}r$ and $\Omega_k,D_k \searrow \Omega',D'$ respectively. Fix two functions $\gamma,\rho \in C^\infty(\R)$ such that 
\begin{equation}
\begin{aligned}
0\le \gamma' \le 2, \ \ \ \gamma(t)=
\begin{cases}
0 \ \ \text{if} \ \ t \le 0\\
1 \ \ \text{if} \ \ t \ge 1
\end{cases}, \ \
-2\le \rho'\le 0, \ \ \ \rho(s)=
\begin{cases}
1 \ \ \text{if} \ \ s \le 0\\
0 \ \ \text{if} \ \ s \ge 1
\end{cases}.
\end{aligned}
\end{equation}
Set $\gamma_k(t)\coloneqq \gamma(\frac{t-t_{k-1}}{t_k-t_{k-1}})$ and $\rho_k(s)\coloneqq \rho(\frac{s-r_k}{r_{k-1}-r_k})$, and consider the sequence of cut-off functions given by
\begin{equation}
\eta_k(x,t) \coloneqq \gamma_k(t) \rho_k(d_{g(T)}(x,p)).
\end{equation}
Clearly, $\eta_k$ is a smooth function such that $0\le \eta_k \le 1$, $\eta_k \equiv 1$ on $D_k$ and $\eta_k \equiv 0$ outside $D_{k-1}$. The lower Ricci bound together with the Ricci flow equation gives us
\begin{equation}
\big\lvert \frac{\partial \eta_k}{\partial t} \big\rvert \le \frac{2}{3T}4^{k+1}, \ \ \text{and} \ \ |\nabla \eta_k|_{g(t)} \le e^{2BT} 2^{k+2}r^{-1} \ \ \forall t \in [0,T].
\end{equation}
If $\lambda \ge 2$, recall $\Lambda(\lambda)=6\lambda$, we compute from (\ref{moserabsorbed})
\begin{equation}
\begin{aligned}
\mix{v^{\lambda}}{\alpha',\Omega_k}{\beta',[t_k,T]} &= \mix{\eta_k^2 v^{\lambda}}{\alpha',\Omega_k}{\beta',[t_k,T]}\le \mix{\eta_k^2 v^{\lambda}}{\alpha',\Omega_{k-1}}{\beta',[t_{k-1},T]} \\
&\le C_2(n,a,b,\sigma,C_0) \lambda^{1+\nu} \int_{D_{k-1}}{\Big(|\nabla \eta_k|^2
+2\eta_k\frac{\partial \eta_k}{\partial t}+\eta_k^2 \Big) v^{\lambda} d\mu dt}\\
&\le 4^{k+2}C_3(r,T,B)C_2(n,a,b,\sigma,C_0) \lambda^{1+\nu} \int_{D_{k-1}}{ v^{\lambda} d\mu dt} \\
&\le C_4(n,a,b,\sigma,C_0,r,T,B)4^{k-1}\lambda^{1+\nu} \norm{v^{\lambda}}_{1,D_{k-1}}.
\end{aligned}
\end{equation}
Let us remark that $C_3(r,T,B)$ can be chosen to be $\max \{ e^{4BT}/r^2,1/(3T),1\}$.
It is convenient to rewrite this as
\begin{equation}
\mix{v}{\alpha'\cdot \lambda,\Omega_k}{\beta' \cdot \lambda,[t_k,T]} \le C_4^{\frac{1}{\lambda}} 4^{\frac{k-1}{\lambda}} \lambda^{\frac{1+\nu}{\lambda}} \norm{v}_{\lambda,D_{k-1}}.
\end{equation}
Let us set $\zeta=\min \{\alpha',\beta'\}$, $\alpha_k=\zeta^{k-1} \alpha'$, and $\beta_k=\zeta^{k-1} \beta'$. Applying the inequality above with $\lambda=\zeta^{k-1}$ we obtain
\begin{equation}
\mix{v}{\alpha_k,\Omega_k}{\beta_k,[t_k,T]} \le C_4^{\frac{1}{\lambda}} 4^{\frac{k-1}{\lambda}} \lambda^{\frac{1+\nu}{\lambda}} \norm{v}_{\zeta^{k-1},D_{k-1}}\le C_4^{\frac{1}{\lambda}} 4^{\frac{k-1}{\lambda}} \lambda^{\frac{1+\nu}{\lambda}} \mix{v}{\alpha_{k-1},\Omega_{k-1}}{\beta_{k-1},[t_{k-1},T]}.
\end{equation}
Using $\lambda=\zeta^{k-1},\zeta^{k-2},...,\zeta^{k_0}$, where $k_0$ is the smallest integer such that $\zeta^{k_0} \ge 2$, a simple iteration implies
\begin{equation}
\begin{aligned}
\mix{v}{\alpha_k,\Omega_k}{\beta_k,[t_k,T]} \le& C_4^{\frac{1}{\zeta^{k-1}}+\frac{1}{\zeta^{k-2}}+...+\frac{1}{\zeta^{k_0}}} 4^{\frac{k-1}{\zeta^{k-1}}+\frac{k-2}{\zeta^{k-2}}+...+\frac{k_0}{\zeta^{k_0}}} \zeta^{(1+\nu)\big( \frac{k-1}{\zeta^{k-1}}+\frac{k-2}{\zeta^{k-2}}+...+\frac{k_0}{\zeta^{k_0}} \big)} \cdot\\
&\cdot\mix{v}{\alpha_{k_0},\Omega_{k_0}}{\beta_{k_0},[t_{k_0},T]}\le C_5(n,a,b,\sigma,C_0,r,T,B) \mix{v}{\alpha_{k_0},\Omega_{k_0}}{\beta_{k_0},[t_{k_0},T]},
\end{aligned}
\end{equation}
where in the last step we used that any of the exponents in consideration can be bounded by a summable series, whose value is independent of $k$.
To cover the left cases $\lambda<2$, if needed we can iterate directly (\ref{moserabsorbed}) (a finite amount of times independent of $k$) with the different definition of $\Lambda(\lambda)$, and we obtain
\begin{equation}
\mix{v}{\alpha_{k_0},\Omega_{k_0}}{\beta_{k_0},[t_{k_0},T]}\le C_6(n,a,b,\sigma,C_0,r,T,B) \mix{v}{\alpha',\Omega_0}{\beta',[\frac{T}{4},T]}.
\end{equation}
Resuming, we have
\begin{equation}
\mix{v}{\alpha'_k,\Omega_k}{\beta'_k,[t_k,T]} \le C_7(n,a,b,\sigma,C_0,r,T,B) \mix{v}{\alpha',\Omega_0}{\beta',[\frac{T}{4},T]},
\end{equation}
therefore from the inclusions of the domains
\begin{equation}
\mix{v}{\alpha'_k,\Omega'}{\beta'_k,[\frac{T}{2},T]} \le \mix{v}{\alpha'_k,\Omega_k}{\beta'_k,[t_k,T]} \le C_7 \mix{v}{\alpha',\Omega_0}{\beta',[\frac{T}{4},T]} \le C_7 \mix{v}{\alpha',\Omega}{\beta',[0,T]}.
\end{equation}
With $k$ going to infinity, both $\alpha'_k$ and $\beta'_k$ tend to infinity, so we obtain
\begin{equation}
\norm{v}_{\infty,D'} \le C \mix{v}{\alpha',\Omega}{\beta',[0,T]},
\end{equation}
using Lemma \ref{parameterstoinfty}.
Notice that the conditions in this Lemma are satisfied by the domain $\Omega'$, because the flow is smooth on the whole time interval. Since $u\ge 0$, using the definition of $v$ we finally get
\begin{equation}
\begin{aligned}
\norm{u}_{\infty,D'} &\le \norm{v}_{\infty,D'} \le C \mix{v}{\alpha',\Omega}{\beta',[0,T]} \le C (\mix{u}{\alpha',\Omega}{\beta',[0,T]}+\kappa \mix{1}{\alpha',\Omega}{\beta',[0,T]})\\
&\le C (\mix{u}{\alpha',\Omega}{\beta',[0,T]}+\mix{h}{a,\Omega}{b,[0,T]} \mix{1}{\alpha',\Omega}{\beta',[0,T]}).
\end{aligned}
\end{equation}
\end{proof}
\end{lemma}
\begin{remark}\label{mosercauchy}
We remark that, following the same iteration scheme, we could arrive at
\begin{equation}
\norm{v}_{\infty,D'}\le C_7 \mix{v}{\alpha_l,\Omega}{\beta_l,[0,T]}
\end{equation}
for every $l \in \mathbb{N}$. In particular, choosing $l$ large enough, we can get $\alpha_l\ge\alpha$ and $\beta_l\ge\beta$, hence
\begin{equation}
\norm{v}_{\infty,D'}\le C_7 \mix{v}{\alpha,\Omega}{\beta,[0,T]}.
\end{equation}
\end{remark}
\subsection{Improved Integrability and Moser-Harnack Inequality in the Optimal Case}
In showing a Moser-Harnack type inequality for the scalar curvature in the optimal case, we cannot directly appeal to the absorption scheme used in Lemma \ref{supermoser}, but we need to impose a certain smallness of the data, see Theorem \ref{optimalmoser}; this feature is common among differential equations with super-linear forcing terms, where one can exploit a point-wise smallness of the solution to reduce the study to the linear case. Here the smallness assumption is given in an integral form rather than point-wise, thus we develop once again a Moser's iteration, this time to link the problem to the strictly super-optimal case, compare with Lemma $4.4$ in \cite{wan}. Furthermore, the proof of the following proposition directly implies an improved integrability result, see Remark \ref{integrabilityaswish}.
\begin{proposition}\label{betterintegrability}
Given $(M,g(t))$ as above, suppose there exists a constant $B \ge 0$ such that on $D$ we have $\Ric(x,t) \ge -Bg(t)$ and let $(\alpha,\beta)$ be an optimal couple. If a non-negative function $u \in C^\infty(D)$ satisfies
\begin{equation}
\frac{\partial u}{\partial t} \le \Delta u+fu+h
\label{equazione2}
\end{equation}
in the sense of distributions, where $\mix{f}{\alpha,\Omega}{\beta,[0,1]}<+\infty$, then there exist a strictly super-optimal couple $(a,b)$ and constants $C=C(n,\alpha,\beta,\sigma,r,T,B)$ and $\delta=\delta(n,\sigma,\alpha,\beta)$ such that if the smallness assumption $\mix{f}{\alpha,\Omega}{\beta,[0,T]}+\mix{\Sc_{-}}{\alpha,\Omega}{\beta,[0,T]} \le \delta$ is satisfied, then we have
\begin{equation}
\mix{u}{a,\Omega'}{b,[\frac{T}{2},T]} \le C(\mix{u}{\alpha,\Omega}{\beta,[0,T]}+\mix{h}{\alpha,\Omega}{\beta,[0,T]}\cdot \mix{1}{\alpha,\Omega}{\beta,[0,T]}).
\end{equation}
\begin{proof}
Let $v=u+\kappa$, where $\kappa=l\cdot \mix{h}{\alpha,\Omega}{\beta,[0,T]}$ for some positive constant $l$. Then $v$ solves
\begin{equation}
\partial_t v \le \Delta v+fu+h.
\end{equation}
We choose a test function of the form $\eta^2 v^{\lambda-1}$, and proceed as in the proof of Lemma \ref{supermoser} to get
\begin{equation}
\begin{aligned}
\mix{\eta^2 v^{\lambda}}{\alpha',\Omega}{\beta',[0,T]} \le& \sigma^{\frac{1}{\beta'}}\Lambda(\lambda) \bigg( \Big( \mix{f}{\alpha,\Omega}{\beta,[0,T]}+\mix{\Sc_{-}}{\alpha,\Omega}{\beta,[0,T]}+\frac{1}{l}\Big) \mix{\eta^2 v^{\lambda}}{a',\Omega}{b',[0,T]}\\
&+\left\|\Big(|\nabla \eta|^2
+2\eta\frac{\partial \eta}{\partial t} \Big) v^{\lambda} \right\|_{1,D} \bigg),
\end{aligned}
\end{equation}
Suppose now that $\mix{f}{\alpha,\Omega}{\beta,[0,T]}+\mix{\Sc_{-}}{\alpha,\Omega}{\beta,[0,T]} \le \delta_\lambda \coloneqq (4 \sigma^{\frac{1}{\beta'}} \Lambda(\lambda))^{-1}$; in order to absorb the first term of the right hand side, we choose $l_\lambda \coloneqq 4 \sigma^{\frac{1}{\beta'}} \Lambda(\lambda)+1$, so we get
\begin{equation}
\mix{\eta^2 v^{\lambda}}{\alpha',\Omega}{\beta',[0,T]} \le 2 \sigma^{\frac{1}{\beta'}} \Lambda(\lambda)\left\|\Big(|\nabla \eta|^2
+2\eta\frac{\partial \eta}{\partial t} \Big) v^{\lambda} \right\|_{1,D}.
\label{prechoice}
\end{equation}
We have arrived at a situation analogous to (\ref{moserabsorbed}).
Choosing properly the function $\eta$ we get a reverse H\"older inequality, so a better integrability. The inequality above ensures that, given the optimal couple $(\alpha,\beta)$, we can bound the $(\alpha' \lambda,\beta' \lambda)-$norm in terms of the $(\alpha,\beta)-$norm if we have $\lambda \le \min \{ \alpha, \beta \}$. On the other hand, the condition on the couple $(\alpha' \lambda,\beta' \lambda)$ to be strictly super-optimal is
\begin{equation}
\alpha' \lambda>\frac{n}{2} \frac{\beta' \lambda}{\beta' \lambda -1} \iff \lambda>\frac{n}{2}.
\end{equation} 
These conditions imply that $\min \{ \alpha, \beta \} \eqqcolon \zeta>\frac{n}{2}$, which is restrictive. In order to cover the general case, we consider the same nested family of cylinders and cut-off functions as in Lemma \ref{supermoser}. Set $\alpha_k=(\alpha')^k \zeta$ and $\beta_k=(\beta')^k \zeta$ for every $k$. Iterating the inequality (\ref{prechoice}) with $\lambda=\alpha_{k-1}>\alpha_{k-2}>...>\alpha_{1}>\beta$ if $\beta<\alpha$, or $\lambda=\beta_{k-1}>\beta_{k-2}>...>\beta_{1}>\alpha$ if $\beta>\alpha$, we obtain
\begin{equation}
\mix{v}{\alpha_k,\Omega_k}{\beta_k,[t_k,T]}\le C \mix{v}{\alpha_{k-1},\Omega_{k-1}}{\beta_{k-1},[t_{k-1},T]} \le ... \le C\mix{v}{\alpha'\zeta,\Omega_1}{\beta'\zeta,[t_1,T]}\le C \mix{v}{\alpha,\Omega}{\beta,[0,T]}.
\end{equation}
Here the value of $C$ varies from inequality to inequality, but it never approaches infinity. Since $D' \subset D_k$ for every $k$, in order to conclude it suffices to show that for the fixed couple $(\alpha,\beta)$, there exists a finite $k$ such that $(\alpha_k,\beta_k)$ is super-optimal. By definition, we need to check
\begin{equation}
(\alpha')^k\zeta>\frac{n}{2} \frac{(\beta')^k\zeta}{(\beta')^k\zeta-1}.
\end{equation}
Since the limit for $k \rightarrow +\infty$ of the left hand side is infinite, we deduce the existence of a finite $k$ satisfying this inequality, and hence the claim.
\end{proof}
\label{premoser}
\end{proposition}
\begin{remark}[Improved Integrability]\label{integrabilityaswish}
Notice that the above iteration process does not reach the endpoint case $(\alpha,\beta)=(\infty,1)$ unless $n=1$. 
Moreover, by taking a further larger $k$ we may assume $a$ and $b$ to be as large as we want.
\end{remark}
We conclude the subsection using the integrability result just obtained to deduce a Moser-Harnack inequality for the scalar curvature.
\begin{theorem}\label{optimalmoser}
Under the above assumptions there exist a constant $C=C(n,\alpha,\beta,\sigma,r,T,B)$ and a small constant $\delta=\delta(n,\sigma,\alpha,\beta)$, such that if $\mix{\Sc}{\alpha,\Omega}{\beta,[0,T]} +B \le \delta$, then we have
\begin{equation}\label{smallness}
\norm{\Sc_{+}}_{\infty,D'} \le C( \mix{\Sc}{\alpha,\Omega}{\beta,[0,T]}+B) \le C \delta.
\end{equation}
\begin{proof}
Set $\hat{\Sc} \coloneqq \Sc+nB$. By the lower bound on the Ricci tensor, with same argument as in \cite{wan}, we get the following inequality in $D$:
\begin{equation}
\frac{\partial \hat{\Sc}}{\partial t} \le \Delta \hat{\Sc} + 2(\hat{\Sc}-2B) \hat{\Sc}+2nB^2.
\end{equation}
Remark that we are in case where $\mix{1}{\alpha,\Omega}{\beta,[0,T]}\le \tilde{V}$ by Lemma \ref{volumebound}. Referring to Proposition \ref{premoser}, we set $u=\hat{\Sc}$, $f=2(\hat{\Sc}-2B)$, $h=2nB^2$, and we call $C'$ and $\delta'$ the constants given by the Proposition. Let $\delta \coloneqq \frac{\delta'}{3n\tilde{V}}$, and compute
\begin{equation}
\begin{aligned}
\mix{f}{\alpha,\Omega}{\beta,[0,T]}+\mix{\Sc_{-}}{\alpha,\Omega}{\beta,[0,T]}&=\mix{2(\hat{\Sc}-2B)}{\alpha,\Omega}{\beta,[0,T]}+\mix{\Sc_{-}}{\alpha,\Omega}{\beta,[0,T]}\\
&=\mix{2(\Sc+(n-2)B)}{\alpha,\Omega}{\beta,[0,T]}+\mix{\Sc_{-}}{\alpha,\Omega}{\beta,[0,T]}\\
&\le 3\mix{\Sc}{\alpha,\Omega}{\beta,[0,T]}+2(n-2)B\mix{1}{\alpha,\Omega}{\beta,[0,T]}\\
&\le 3n\tilde{V}(\mix{\Sc}{\alpha,\Omega}{\beta,[0,T]}+B)\le \delta'.
\end{aligned}
\end{equation}
We can apply Proposition \ref{premoser} to get the existence of a super-optimal couple $(a,b)$ such that
\begin{equation}
\mix{\hat{\Sc}}{a,\Omega'}{b,[\frac{T}{2},T]} \le C (\mix{\hat{\Sc}}{\alpha,\Omega}{\beta,[0,T]}+\mix{2nB^2}{\alpha,\Omega}{\beta,[0,T]}\mix{1}{\alpha,\Omega}{\beta,[0,T]}) \le C.
\label{supersmall}
\end{equation}
We can bound 
\begin{align}
\mix{2(\hat{\Sc}-2B)}{a,\Omega'}{b,[\frac{T}{2},T]}+\mix{\Sc_{-}}{a,\Omega'}{b,[\frac{T}{2},T]} +1 &\le 3 \mix{\hat{\Sc}}{a,\Omega'}{b,[\frac{T}{2},T]}+(n+4)B\mix{1}{a,\Omega'}{b,[\frac{T}{2},T]}+1 \nonumber\\
&\le C_0(n,\alpha,\beta,\sigma,r,T,B).
\end{align}
Notice that we can assume $\alpha_*'\le a$ and $\beta_*' \le b$ by Remark \ref{integrabilityaswish}, where $\alpha_*$ and $\beta_*$ are the exponents given by Lemma \ref{extrapolationlemma}, for which (\ref{equazione2}) holds. Now we apply the super-optimal Moser iteration Lemma \ref{supermoser} to get the existence of a constant $C=C(n,\alpha,\beta,\sigma,r,T,B)$ such that (we use H\"older's inequality, (\ref{supersmall}), $\alpha_*'\le a,\beta_*' \le b$ and the definition of $\hat{\Sc}$)
\begin{equation}
\begin{aligned}
\norm{\hat{\Sc}}_{\infty,D'} &\le C(\mix{\hat{\Sc}}{\alpha_*',\Omega}{\beta_*',[0,1]}+\mix{h}{a,\Omega}{b,[0,1]}\cdot \mix{1}{\alpha_*',\Omega}{\beta_*',[0,1]})\\
&=C(\mix{\hat{\Sc}}{\alpha_*',\Omega}{\beta_*',[0,1]}+2nB^2 \mix{1}{a,\Omega}{b,[0,1]}\cdot \mix{1}{\alpha_*',\Omega}{\beta_*',[0,1]})\\
&\le C(C \mix{\hat{\Sc}}{a,\Omega}{b,[0,1]}+2nB^2 \mix{1}{a,\Omega}{b,[0,1]}\cdot \mix{1}{\alpha_*',\Omega}{\beta_*',[0,1]})\\
&\le C(\mix{\Sc}{\alpha,\Omega}{\beta,[0,1]}+B).
\end{aligned}
\end{equation}
It suffices to notice that $\norm{\Sc_{+}}_{\infty,D'} \le \norm{\hat{\Sc}}_{\infty,D'}$ to conclude the proof.
\end{proof}
\end{theorem}
\subsection{Proof of Theorem \ref{extensiontheorem2}}
In this subsection we use the results of the previous sections to give a proof of Theorem \ref{extensiontheorem2}. Inspired by the proof of Theorem \ref{extensionth1} we argue by contradiction, assuming that the flow is not extensible, and deduce a contradiction from an asymptotic analysis for a sequence of rescalings. The hypotheses assumed in the statement of Theorem \ref{extensiontheorem2} naturally lead to rescale the flow with the scalar curvature. Were we rescaling at the maximal Riemann curvature scale, we would have the right bounds needed to extract a blow-up limit, hence we could get a contradiction similar to the one obtained in the proof of Theorem \ref{extensionth1}. Unfortunately, the choice of the scaling factors introduces the technical problem of the lack of compactness for the sequence. In order to deal with this issue, we apply the Moser-Harnack inequality given by Theorem \ref{optimalmoser}, which holds uniformly for the sequence of rescalings considered in view of the assumed lower Ricci bound; in fact, the uniformity of the Sobolev constant is guaranteed by Theorem \ref{uniformsobolevth}, and we consider appropriate parabolic regions for the inequality, see below. Finally, we deduce from the finiteness of the mixed integral norm considered that the sequence of rescalings uniformly approaches scalar flatness in a region containing a normalized scalar curvature point, a contradiction.
\begin{proof}
Without loss of generality, we can assume that $(\alpha,\beta)$ is an optimal couple: indeed similar to what we did in the proof of Theorem \ref{extensionth1}, it is sufficient to apply H\"older's inequality in time to reduce the problem to the optimal case.

Suppose by contradiction that the flow cannot be extended; thus Theorem $1.4$ in \cite{chm} implies that $|\Ric|$ is unbounded on $M \times [0,T)$. The assumed lower Ricci bound implies the unboundedness of the scalar curvature, $\sup_{M \times [0,T)} \Sc =+\infty$. Since the curvature tensor is bounded up to the singular time $T$, we can pick a sequence of space-time points $(x_i,t_i)$ such that $t_i \nearrow T$, and
\begin{equation}
\Sc(x_i,t_i)\ge C^{-1} \sup_{M\times [0,t_i]} \Sc(x,t).
\end{equation}
Here $C$ is any constant greater than $1$. Set $Q_i \coloneqq \Sc(x_i,t_i) \rightarrow +\infty$ and $P_i \coloneqq B_{g(t_i)}(x_i,Q_i^{-\frac{1}{2}}) \times [t_i-Q_i^{-1},t_i]$. Clearly, $\Sc \le C Q_i$ on the parabolic region $P_i$. Consider a sequence of Ricci flows on $M \times [0,1]$ defined as $g_i(t)\coloneqq Q_ig(Q_i^{-1}(t-1)+t_i)$. By construction we get
\begin{equation}
\begin{aligned}
\Sc_i(x,t)&\le C \ \ \ \forall (x,t) \in B_{g_i(1)}(x_i,1)\times [0,1] \ \ \text{and} \ \ \Sc_i(x_i,1)=1;\\
\Ric_i(x,t)&\ge -\frac{B}{Q_i} g_i(t) \ \ \ \ \ \forall (x,t) \in M \times [0,1].
\end{aligned}
\end{equation}
From this we deduce that the tensor $\Ric_i+\frac{B}{Q_i}$ is non-negative, thus
\begin{equation}
\Ric_i+\frac{B}{Q_i} \le tr(\Ric_i+\frac{B}{Q_i})g_i= \big( \Sc_i+\frac{nB}{Q_i}\big)g_i.
\end{equation}
In particular, choosing $C \le n-2$, we get for $i$ large enough
\begin{equation}
\begin{aligned}
\Ric_i(x,t)&\le (n-1) g_i(t) \ \ \ \forall (x,t) \in B_{g_i(1)}(x_i,1)\times [0,1];\\
\Ric_i(x,t)&\ge -\frac{B}{Q_i}g_i(t) \ \ \ \ \ \ \forall (x,t) \in M \times [0,1].
\end{aligned}
\end{equation}
Moreover, for $i$ large enough, $-nB \le \Sc(x,t) \le C Q_i$ on $P_i$ implies $|\Sc_i| \le C$ on $B_{g_i(1)}(x_i,1)$, so the $\kappa$-non local collapsing theorem Theorem \ref{nonlocalcollapsingtheorem} applies (with scale $2$) yielding for every scale $\rho$ the existence of a constant $\kappa=\kappa(g(0),n,T)$ such that we have the uniform lower bound
\begin{equation}
\Vol_{g_i(1)}(B_{g_i(1)}(x_i,1))= \frac{\Vol_{g(t_i)}(B_{g(t_i)}(x_i,Q_i^{-\frac{1}{2}}))}{Q_i^{-\frac{n}{2}}}\ge \kappa.
\end{equation}
Theorem \ref{uniformsobolevth} now guarantees the existence of a radius $r=r(\kappa,n)$ and a uniform Sobolev constant $\sigma=\sigma(n,r)$ for every time-slice $t \in [0,1]$ on the ball $B_{g_i(1)}(x_i,r)$. We can therefore use Theorem \ref{optimalmoser}, by setting
\begin{equation}
\begin{aligned}
\Omega_i \coloneqq B_{g_i(1)}(x_i,r),& \ \ \ \Omega_i' \coloneqq B_{g_i(1)}(x_i,\frac{r}{2}),\\
D_i \coloneqq \Omega_i \times [0,1],& \ \ \ D_i' \coloneqq \big[ \frac{1}{2},1\big].
\end{aligned}
\end{equation}
In the following we exploit the finiteness of $\mix{\Sc}{\alpha,M}{\beta,[0,T)}$, as well as its scaling invariance in the case $(\alpha,\beta)$ is an optimal couple, to compute
\begin{equation}
\begin{aligned}
\lim_{i \rightarrow +\infty}{\mix{\Sc}{\alpha,B_{g_i(1)}(x_i,r)}{\beta,[0,1)}^\beta}&=\lim_{i \rightarrow +\infty} \int_{0}^1{ \bigg( \int_{\Omega_i}{|\Sc_{g_i(t)}|^{\alpha} d\mu_{g_i(t)}}\bigg)^{\frac{\beta}{\alpha}} dt}\\
&=\lim_{i \rightarrow +\infty} \int_{t_i-Q_i^{-1}}^{t_i}{ \bigg( \int_{B_{g(t_i)}(x_i,rQ_i^{-\frac{1}{2}})}{|\Sc_{g(t)}|^{\alpha} d\mu_{g(t)}}\bigg)^{\frac{\beta}{\alpha}}dt}=0.\\
\end{aligned}
\end{equation}
Since $\mix{\Sc}{\alpha,B_{g_i(1)}(x_i,r)}{\beta,[0,1)}+\frac{B}{Q_i} \le \delta$ for $i$ large, where $\delta=\delta(n,\sigma,\alpha,\beta)$ is the one given by Theorem \ref{optimalmoser}, we can apply the theorem to obtain
\begin{equation}
\norm{(\Sc_i)_{+}}_{\infty,D_i'} \le C(n,\alpha,\beta,\sigma)( \mix{\Sc_i}{\alpha,\Omega_i}{\beta,[0,1]}+\frac{B}{Q_i}).
\end{equation}
Let us remark, that here we dropped the dependence of the constant $C$ given by Theorem \ref{optimalmoser} on the lower bound $\frac{B}{Q_i}$ because this last one can be uniformly bounded by any constant asymptotically. However, the points $x_i$ were selected in such a way that $\norm{(\Sc_i)_{+}}_{\infty,D_i'}\ge\Sc_i(x_i,1)=1$, so the inequality just obtained gives a contradiction for $i$ large enough.
\end{proof}
\appendix
\counterwithin{equation}{section}
\section{Proofs of Some Lemmas}
The goal of this appendix is to prove Lemma \ref{parameterstoinfty}. Before doing so, we show an analogous result on the averages.
\begin{lemma}
Let $(M,g(t))$ be a Ricci flow defined in $[0,T]$, and fix a subset $\Omega' \subset M$ such that $0<c \le \Vol_{g(t)}(\Omega')\le C<+\infty$ for every $t$. Then we have for any measurable $u$
\begin{equation}
\lim_{(a,b) \rightarrow (+\infty,+\infty)}{\phi(a,b)}=\sup_{\Omega' \times [0,T]}{|u|(x,t)},
\end{equation}
where
\begin{equation}
\phi(a,b) \coloneqq \bigg( \frac{1}{T}\int_0^T{ \bigg( \frac{1}{\Vol_{g(t)}(\Omega')}\int_{\Omega'}{|u|^a d\mu_{g(t)}}\bigg)^{\frac{b}{a}} dt} \bigg)^{\frac{1}{b}}.
\end{equation}
\begin{proof}
Without loss of generality we can assume $u \ge 0$ and $u\in\mathbb{L}^\infty$ (otherwise we prove it for the truncation $u_M \coloneqq \max{ \{ \min{ \{u,M\} },-M\} }$ and then let $M$ to infinity). A simple application of H\"older inequality gives that $\phi$ is a non-decreasing function of both $a$ and $b$. Set $D'\coloneqq \Omega'\times[0,T]$. We can bound $u$ from above with its essential supremum, so that we obtain 
\begin{equation}
\lim_{(a,b) \rightarrow (+\infty,+\infty)}{\phi(a,b)} \le \sup_{D'}{u(x,t)}.
\end{equation}
For any $\e>0$, by definition of essential supremum we get the existence of a $\delta>0$ such that, if we set $E \coloneqq \{ (x,t) \in D' \mid u(x,t) \ge \sup_{D'}u-\e \}$, we have
\begin{equation}
|E| \coloneqq \int_0^T{\bigg( \int_{E_t} d\mu_{g(t)}\bigg) dt}>\delta,
\end{equation}
where $E_t \coloneqq E \cap (M \times \{ t \})$. Therefore we obtain:
\begin{equation}
\begin{aligned}
\phi(a,b) &\ge \bigg( \frac{1}{T}\int_0^T \bigg( \frac{1}{\Vol_{g(t)}(\Omega')}\int_{E_t}{(\sup_{D'} u-\e)^a d\mu_{g(t)}}\bigg)^{\frac{b}{a}} dt \bigg)^{\frac{1}{b}} \label{anything2}\\
&=\bigg( \frac{1}{T}\int_0^T{ \bigg( \frac{\Vol_{g(t)}(E_t)}{\Vol_{g(t)}(\Omega')}\bigg)^{\frac{b}{a}} (\sup_{D'} u-\e)^b dt} \bigg)^{\frac{1}{b}}
\ge (\sup_{D'} u-\e) \frac{1}{C^{\frac{1}{a}}} \bigg( \frac{1}{T}\int_0^T{ \Vol_{g(t)}(E_t)^{\frac{b}{a}}dt} \bigg)^{\frac{1}{b}}.
\end{aligned}
\end{equation}
Fubini's theorem implies that if it was $\Vol_{g(t)}(E_t)=0$ for a.e. $t \in [0,T]$, we would have $|E_t|=0$ which is contradictory. Thus there exist a set $A \subset [0,T]$, with $|A|=\delta'>0$, and numbers $\delta_t>0$ such that $f(t)\coloneqq \Vol_{g(t)}(E_t) \ge \delta_t$ for every $t \in A$. From Lusin's theorem for every $\e'$ we deduce the existence of $A' \subset A$, which we can assume to be closed and hence compact, with $|A \setminus A'| \le \e'$ and $f$ continuous in $A'$, so that $f \ge \delta_{\e'}$ on $A'$. We can thus deduce
\begin{equation}
\begin{aligned}
\phi(a,b) \ge& (\sup_{D'} u-\e) \frac{1}{C^{\frac{1}{a}}} \bigg( \frac{1}{T}\int_0^T{ \Vol_{g(t)}(E_t)^{\frac{b}{a}}dt} \bigg)^{\frac{1}{b}} \\
\ge& (\sup_{D'} u-\e) \frac{1}{C^{\frac{1}{a}}} \bigg( \frac{1}{T}\int_{A'}{ \Vol_{g(t)}(E_t)^{\frac{b}{a}} dt} \bigg)^{\frac{1}{b}}\ge (\sup_{D'} u-\e) \frac{1}{\tilde{C}^{\frac{1}{a}}} |A'|^{\frac{1}{b}}.
\end{aligned}
\end{equation}
It suffices now to first let $a$ and $b$ go to infinity, and then $\e$ to zero to conclude the proof.
\end{proof}
\end{lemma}
\begin{proof}[Proof of Lemma \ref{parameterstoinfty}]
We easily compute
\begin{equation}
\begin{aligned}
\mix{u}{a,\Omega'}{b,[0,T]}&= \bigg(T \fint_0^T{ \Vol_{g(t)}(\Omega')^{\frac{1}{a}} \fint_{\Omega'}{|u|^a d\mu}\bigg)^{\frac{b}{a}} dt} \bigg)^{\frac{1}{b}} \\
&\le T^{\frac{1}{b}} (\sup_{[0,T]}{\Vol_{g(t)}(\Omega')})^{\frac{1}{a}} \phi(a,b)\le T^{\frac{1}{b}} C^{\frac{1}{a}} \phi(a,b).
\end{aligned}
\end{equation}
Similarly, we get
\begin{equation}
\mix{u}{a,\Omega'}{b,[0,T]}\ge T^{\frac{1}{b}} c^{\frac{1}{a}} \phi(a,b).
\end{equation}
The conclusion follows taking the limit for $a$ and $b$ going to infinity.
\end{proof}

\end{document}